\documentclass[10pt,a4paper,oneside]{amsproc}

\usepackage{latexsym, amsthm}
\usepackage{amsfonts, amsmath, amssymb}
\usepackage{euscript,mathrsfs}
\usepackage{url}
\usepackage{array}
\usepackage[a4paper]{geometry}
\usepackage{graphics}
\usepackage[usenames]{color}
\usepackage[all]{xy}
\usepackage{graphicx}
\usepackage{boxedminipage}

\newtheorem{theorem}{Theorem}[section]

\newtheorem{conjecture}{Conjecture}

\newtheorem{question}[theorem]{Question}
\newtheorem{definition}[theorem]{Definition}
\newtheorem{remark}[theorem]{Remark}

\newtheorem{lemma}[theorem]{Lemma}

\newtheorem{proposition}[theorem]{Proposition}

\newcommand{\Fq}{\mathbb{F}_q}

\def\Fq{{\mathbb F}_q}
\def\AA{{\mathbb A}}
\def\FF{{\mathbb F}}
\def\PP{{\mathbb P}}
\def\C{{\mathcal C}}
\def\L{{\mathcal L}}
\def\B{{\mathcal B}}

\def\S{{\mathfrak S}}
\def\X{{\mathcal X}}

\begin{document}

\title[Number of rational points on nonsingular threefolds]{Maximum number of $\Fq$-rational points on nonsingular threefolds in $\PP^4$}

\author{Mrinmoy Datta}
\address{Department of Mathematics and Statistics, 
University of Troms{\o}, 9037 Troms{\o}, Norway.}
\email{mrinmoy.dat@gmail.com}
\thanks{The author is supported by a postdoctoral fellowship from DST-RCN grant INT/NOR/RCN/ICT/P-03/2018. }
%

\begin{abstract}   
We determine the maximum number of $\Fq$-rational points that a nonsingular threefold of degree $d$ in a projective space of dimension $4$ defined over $\Fq$ may contain. This settles a conjecture by Homma and Kim concerning the maximum number of points on a hypersurface in a projective space of even dimension in this particular case. 
\end{abstract}

\date{}
\maketitle

\section{Introduction}
For a prime power $q$, we denote by $\Fq$ a finite field with $q$ elements and by $\bar{\FF}_q$ a fixed algebraic closure of $\Fq$. Let $m, d$ be positive integers. We revisit the question of determining the maximum number of $\Fq$-rational points on a nonsingular hypersurface  defined over $\Fq$ contained in an $m$-dimensional projective space over an algebraic closure of $\Fq$. More specifically we look at the following question:

\begin{question}\label{ques}
Let $\X \subset \PP^m (\bar{\FF}_q)$ be a nonsingular hypersurface of degree $d$ defined over $\Fq$. What is the maximum number of $\Fq$-rational point that $\X$ may have?
\end{question}


From now on, we will restrict our attention to the case when $2 \le d \le q$. If $m=2$, then $\X$ is a nonsingular plane curve defined over $\Fq$ and from the famous Hasse-Weil Theorem, we know that $|\X (\Fq)| \le 1 + q + (d-1)(d-2)\sqrt{q}$ and this bound is attained by the Hermitian curve. 

Recently, Homma and Kim have addressed the Question \ref{ques} and made significant progress towards answering the same. They have proved \cite{HK} the following inequalities:
\begin{equation}\label{odd}
|\X (\Fq)| \le \theta_q \big(\frac{m-1}{2}\big)\big((d-1)q^{\frac{m-1}{2}} + 1\big), \ \ \text{if} \ m \ge 3 \ \text{and} \ m \ \text{is odd},
\end{equation}
and 
\begin{equation}\label{even}
|\X (\Fq)| \le \theta_q \big(\frac{m}{2}\big) (d-1)q ^{\frac{m}{2} - 1} + \theta_q \big(\frac{m}{2} - 1\big), \ \ \text{if} \ m  \ \text{is even},
\end{equation}
where $\theta_q (j) = 1 + q + \dots +q^{j}$ if $j \ge 0$ and $0$ if $j < 0$.
A complete list of hypersurfaces that attain the upper bound in \eqref{odd} is given in (\cite[Theorem 1.1]{HK}). 
However, it turns out that the upper bound in \eqref{even} is never attained (see \cite[Annotation]{HK}). To this end, the following conjectural bound  was proposed \cite[Conjecture]{HK}. 

\begin{conjecture}\label{con}
Suppose $m \ge 4$ is an even integer and $\X \subset \PP^m$ be a hypersurface of degree $d$ defined over $\Fq$. Then 
$$|\X (\Fq)| \le \theta_q \big(\frac{m}{2} - 1\big) \big((d-1)q^{\frac{m}{2}} + 1\big).$$
\end{conjecture}

In Theorem \ref{main}, we prove Conjecture \ref{con} in the case when $m=4$ and $2 \le d \le q$, except when $(d, q) = (4, 4)$. More precisely, we show that if $\X \subset \PP^4$ is a nonsingular threefold  of degree $d$ defined over $\Fq$, then $|\X (\Fq)| \le (d-1)q^3 + (d-1)q^2 + q +1$.

The paper is organized as follows: In Section 2, we recall various upper bounds on the number of $\Fq$-rational points on hypersurfaces defined over $\Fq$. In Section 3, we derive an upper bound on the number of lines contained in a surface each containing a common point of intersection. Finally, in Section 4, we prove our main result. 

\section{Preliminaries}\label{sec:prel}
In this section, we recall some well-known upper bounds on the number of $\Fq$-rational points on a hypersurface defined over $\Fq$ in terms of its degree and dimension. For a positive integer $m$, we will denote by $\PP^m$ (resp. $\AA^m$)  the projective space (resp. affine space) of dimension $m$ over the field $\bar{\FF}_q$, while $\PP^m (\Fq)$ (resp. $\AA^m (\Fq)$)  will denote the set of all $\Fq$-rational points in $\PP^m$(resp. $\AA^m$).  Given a variety $\X$, we will denote by $\X(\Fq)$ the set of its $\Fq$-rational points. We recall an optimal upper bound for the number of $\Fq$-rational points on an affine hypersurface defined over $\Fq$. We also record, for ease of reference, a result by Geil \cite{G} concerning the second highest number of $\Fq$-rational points on an affine hypersurface defined over $\Fq$. 

\begin{theorem}\label{O}
Let $X \subset \AA^m$ be an affine hypersurface of degree $d$ defined over $\Fq$.  
\begin{enumerate}
\item[(a)] \cite[Thm. 6.13]{LN} if $1 \le d \le q$ then $|X(\Fq)| \le dq^{m-1}$, and 
\item[(b)]\cite[Prop. 2]{G} if $2 \le d \le q-1$ and $|X(\Fq)| < dq^{m-1}$ then $|X(\Fq)| \le dq^{m-1} - (d-1)q^{m-2}$.
\end{enumerate}
\end{theorem}


The following result, concerning the maximum number of $\Fq$-rational points on a projective hypersurface defined over $\Fq$, was proved by Serre \cite{S} and independently by S{\o}rensen \cite{So}. 

\begin{theorem}[Serre-S{\o}rensen]\label{SS}
Let $\X \subset \PP^m$ be a hypersurface of degree $d$ defined over $\Fq$. If $d \le q$ then
$$|\X (\Fq)| \le \mathscr{S} (d, m) = dq^{m-1} + \theta_q (m-2).$$
Further, the bound is attained by a hypersurface $\X$ if and only if $\X$ is a union of $d$ hyperplanes defined over $\Fq$, each containing a common codimension $2$ linear subspace defined over $\Fq$.   
\end{theorem}

We also recall a result by Homma and Kim, referred to as the \emph{elementary bound} \cite[Theorem 1.2]{ELE} concerning the number of $\Fq$-rational points on a hypersurface defined over $\Fq$ that does not contain a $\Fq$-linear component.

\begin{theorem}[Homma-Kim]\label{elementary}
Let $\X \subset \PP^m$ be a hypersurface of degree $d$ defined over $\Fq$. If $\X$ has no $\Fq$-linear component, then 
$|\X (\Fq)| \le \mathscr{E}(d, m) = (d-1)q^{m-1} + dq^{m-2} + \theta_q (m-3).$
\end{theorem}

Next, we recall an upper bound on the number of $\Fq$-rational points on a nonsingular hypersurface which is a consequence of Deligne's work \cite{D} towards establishing the Weil conjecture. 

\begin{theorem}[Deligne]\label{del} 
Let $\X \subset \PP^m$ be a nonsingular hypersurface of degree $d$ defined over $\Fq$. Then
$$|\X (\Fq)| \le \mathscr{W} (d, m) = \frac{d-1}{d} \big( (d-1)^m - (-1)^m \big)q^{\frac{m-1}{2}} + \theta_q (m-1).$$
\end{theorem}

\begin{remark}\label{elt}\normalfont
The upper bound $\mathscr{E} (d, m)$ above, often referred to as the \emph{elementary bound}, deserves a few more remarks. First of all, a complete list of hypersurfaces that can attain the bound is known and can be found in \cite{Tir}. It turns out that a hypersurface of degree $d$, with no linear component defined over $\Fq$, attains the elementary bound only if $d = 2, \sqrt{q} + 1  \ \text{or} \ q+1$. More remarkably, we have $\mathscr{E} (d, m) < \mathscr{W} (d, m)$ whenever $d \ge \sqrt{q} + 2$. We refer to \cite[Proposition 4.2]{ELE} for the proof of this fact. 
\end{remark}

Let $\X \subset \PP^m$ be a hypersurface defined over $\Fq$.
We recall that the \emph{Koen Thas invariant} \cite{T} of $\X$, denoted by $k_{\X}$, is given by
$$k_{\X} := \max \{\dim L \mid L \subset \X, L \ \text{is a linear subspace of} \ \PP^m \ \text{defined over} \ \Fq\}.$$ 

We refer to \cite{Tir1} for upper bounds on the number of $\Fq$-points on hypersurfaces depending on the Koen Thas invariant.  The following proposition which is  a direct consequence of \cite[Lemma 2.1]{HK} gives an upper bound on $k_{\X}$ where $\X$ is a nonsingular projective hypersurface. 

\begin{proposition}\label{kt}
Let $\X$ be a nonsingular hypersurface in $\PP^m$. Then $k_{\X} \le \big\lfloor \frac{m-1}{2} \big\rfloor$.
\end{proposition}

We will also use an upper bound on the number of $\Fq$-rational points on a plane curve defined over $\Fq$ that does not contain a line defined over $\Fq$. In a series of three papers \cite{HK1, HK2, HK3}, Homma and Kim proved the following result.

\begin{theorem}\label{szik}
Let $\C$ be a plane curve of degree $d$ defined over $\Fq$ not containing any lines defined over $\Fq$. Then 
$$|\C (\Fq)| \le (d - 1) q + 1,$$
except for the curve defined over $\FF_4$ given by the vanishing set of the quartic polynomial
$$(X + Y + Z)^4 + (XY + YZ + ZX)^2 + XYZ (X+ Y + Z).$$ 
\end{theorem}

It is worth noting that the bound in Theorem \ref{szik} is better than that given by Theorem \ref{elementary} in this case. We conclude this section with a few observations that will  be helpful in the sequel. 

\begin{remark}\label{obs}
\normalfont

\

\begin{enumerate}
\item[(a)] Fix a positive integer $d \le q$. Let $\X \subset \PP^m$ be a hypersurface of degree $d$ defined over $\Fq$. Suppose, $\X$ is given by the vanishing set of a homogeneous polynomial $F \in \Fq[x_0, \dots, x_m]$ with $\deg F = d$. If $\L$ is a linear subspace of $\PP^m$ such that $\L \not\subset \X$, then $F|_{\L} \neq 0$. In particular, if $\X = V(F) \subset \PP^3$ is a surface defined over $\Fq$ and there is a plane $\Pi \subset \PP^3$ with $\Pi \not\subset \X$, then $F|_{\Pi} \neq 0$. Furthermore, the plane curve $\X \cap \Pi$ may contain at most $d$ lines. 


\item[(b)] Let $\X \subset \PP^m$ be a nonsingular hypersurface containing a line $\ell$. If $P \in \ell$ then $\ell \subset T_P (\X)$, where $T_P(\X)$ is the tangent hyperplane to $\X$ at $P$.  
\end{enumerate}
\end{remark}

\section{An upper bound on number of lines passing through a point on a surface}\label{th}
In this section, we prove a fundamental result concerning the number of lines passing through a given point on a surface. This result will turn out to be instrumental in proving the main Theorem of this paper. 
\begin{theorem}\label{three}
Let $Y \subset \PP^3$ be a surface of degree $d$ defined over $\Fq$ and $P \in Y (\Fq)$. 
Then one of the following holds: 
\begin{enumerate}
\item[(a)] $Y$  contains a plane defined over $\Fq$, 
\item[(b)] $Y$ contains a cone over a plane curve defined over $\Fq$ with center at $P$,
\item[(c)] $\#\{ \ell \subset \PP^3 \mid \ell \ \text{is a line such that} \ P \in \ell \subset Y\} \le d(d-1)$. 
\end{enumerate}
\end{theorem}

\begin{proof}
We assume that the conditions (a) and (b) are not satisfied. Let $\Pi \subset \PP^3$ be a plane defined over $\Fq$ that does not pass through $P$. By a suitable linear change of coordinate systems over $\Fq$, we may assume that $P = [1: 0 : 0 : 0]$ and $\Pi = V(x_0)$. We may further assume that $Y = V(F)$, where $F \in \Fq [x_0, x_1, x_2, x_3]$ and $\deg F = d$.  
Write $$F(x_0, x_1, x_2, x_3) = x_0^{d-1} F_1 (x_1, x_2, x_3) + \dots + F_d (x_1, x_2, x_3),$$
where $F_i \in \Fq [x_1, x_2, x_3]$ are homogeneous polynomials of degree $i$ for $i = 1, \dots, d$.
 First note that $F_d \neq 0$, for otherwise $\Pi \subset Y$. Furthermore, $(F_1, \dots, F_{d-1}) \neq (0, \dots, 0)$, since the condition $F_1= \dots = F_{d-1} =0$ implies that $Y$ is a cone over the plane curve $V(x_0, F_d)$, a contradiction to our assumption. Moreover, the polynomials $F_1, \dots, F_d$ are coprime. For otherwise, there exists a polynomial $G \in \Fq[x_1, x_2, x_3]$ such that $G \mid F_i$ for $i = 1, \dots, d$, so $Y$ contains a cone over the plane curve given by $V (x_0, G)$, which  violates our assumption. 

Define $\L^Y (P) := \{\ell \mid \ell \ \text{is a line with} \  P \in \ell \subset Y\}$. 
There is a natural bijection $S \longleftrightarrow \L^Y (P)$, 
where  $S := \big(\bigcup_{\ell \in \L^Y(P)} \ell\big) \cap \Pi$.
Hence,  it is enough to show that $|S| \le d(d-1)$.

We claim that $S = V(F_1, \dots, F_d, x_0)$. Let $Q = [0 : a_1 : a_2 : a_3] \in S$. Since the line $\ell'$ joining $P$ and $Q$ is contained in $Y$, we see that $[t: a_1 : a_2 : a_3] \in Y$ for all $t \in \overline{\FF}_q$. In particular, $F (t, a_1, a_2, a_3) = 0$ for all  $t \in \bar{\FF}_q$. Since $F (T, a_1, a_2, a_3)$ is a polynomial in $T$ of degree at most $d-1$, we must have
$F_i (a_1, \dots, a_m) = 0$ for all $i= 1, \dots, d-1$. Thus, $S \subset V(F_1, \dots, F_d, x_0)$. The converse is trivial. 

We write $F_d = G_1 \cdots G_r$, where $r \ge 1$ and $G_1, \dots, G_r \in \Fq[x_1, x_2, x_3]$ are irreducible polynomials. Since $F_1, \dots, F_r$ are coprime, for each $i = 1, \dots, r$, there exists $1 \le j_i \le d-1$ such that $G_i$ and $F_{j_i}$ are coprime. By Bezout's theorem $|V(x_0, G_i, F_{j_i})| \le \deg G_i \deg F_{j_i} \le \deg G_i (d-1)$. Hence,
\begin{align*}
|S| \le |V(x_0, F_1, \dots, F_d)| &\le \sum_{i=1}^r |V (x_0, F_1, \dots, F_{d-1}, G_i)|  \\
& \le \sum_{i=1}^r |V(x_0, G_i, F_{j_i})|  \\
&\le \sum_{i=1}^r \deg G_i (d-1) = d (d-1).
\end{align*}
This completes the proof. 
\end{proof}

For the purpose of this paper, we have proved the above theorem for the field $\Fq$. However, it is worth mentioning that the proof goes through when $\Fq$ is replaced by an arbitrary field $k$.

\begin{remark}\normalfont
If $d \le q$, then the upper bound of Theorem \ref{three} can not be improved. To see this, we consider the point $P = [1 : 0 : 0 : 0]$ and $Y \subset \PP^3$, the surface given by the polynomial 
$$F = x_0  \prod_{i=1}^{d-1}(x_1 - a_i x_2) - x_3 \prod_{j=1}^{d-1} (x_2 - a_j x_3),$$
where $a_1, \dots, a_{d-1}$ are distinct non-zero elements of $\Fq$. It is clear that $Y$ does not satisfy the conditions (a) and (b) in Theorem \ref{three} and that $\X$ admits exactly $d (d-1)$ lines containing $P$.
\end{remark}

\section{Main result}\label{mt}
Let $d$ be a positive integer with $2 \le d \le q$ and $\X \subset \PP^{4}$ be a nonsingular threefold of degree $d$ defined over $\Fq$. Given a point $P \in \X (\Fq)$, we denote by $\L(P, \X)$ (resp. $\L_q (P, \X)$) the set of lines (resp. the set of lines defined over $\Fq$) $\ell$  satisfying $P \in \ell \subset \X$. Also, for $P \in \X$, we denote by $T_P (\X)$, the tangent hyperplane to $\X$ at $P$. For a line $\ell \subset \PP^4 (\Fq)$, we denote by $\B (\ell)$, the set of all planes $\Pi \subset \PP^4$ defined over $\Fq$ such that $\ell \subset \Pi$.  
If $\ell \subset \PP^4$ is a line defined over $\Fq$, then
$$|\B (\ell)| = q^2 + q + 1  \ \ \ \text{and}  \ \ \ \PP^4 (\Fq) =  \bigcup_{\Pi \in \B (\ell)} \Pi (\Fq).$$
The following proposition, thanks to the well known classification of quadric hypersurfaces \cite{HT} over finite fields, settles the case where $d=2$.

\begin{proposition}\label{quad}
If $d=2$, then $|\X (\Fq)| = q^3 + q^2 + q + 1.$
\end{proposition}

\begin{proof}
It is known that (see, for example, \cite[Chapter 1]{HT}) any non-singular quadric threefold in $\PP^4$ defined over $\Fq$ is a parabolic quadric upto a projective linear transformation which has exactly $q^3 + q^2 + q + 1$ rational points. 
\end{proof}

 Next, we derive an upper bound for the number of $\Fq$-rational points on $\X$ that lies outside the tangent hyperplane to $\X$ at a given point $P$ on $\X$.

\begin{lemma}\label{affine}
For any $P \in \X(\Fq)$, we have $|\X(\Fq) \cap T_P (\X)^C| \le (d-1)q^3$.  
\end{lemma}

\begin{proof}
Let $S$ denote the set of all lines defined over $\Fq$ that pass through $P$ and is not contained in $T_P(\X)$.  It is easy to show that $T_P (\X)^C (\Fq) = \displaystyle{\bigsqcup_{\ell \in S}} (\ell (\Fq) \setminus \{P\})$, which implies that $|\X(\Fq) \cap T_P(X)^C| = \sum_{\ell \in S} |\X(\Fq) \cap  (\ell \setminus \{P\})|$. Clearly, for any $\ell \in S$, we have $\ell \not\subset \X$ implying $$|\X \cap \ell| \le d , \ \ \ \text{and consequently,} \ \ \ |\X \cap (\ell \setminus \{P\}) | \le d-1.$$
Since $|S| = q^3$, we have $|\X(\Fq) \cap T_P(X)^C| \le (d-1)q^3$.
\end{proof}

The above Lemma applies immediately if we can find an $\Fq$-rational point on $\X$ such that  $\L_q(P, \X) = \emptyset$. The following Lemma shows that the conjecture is true in such a case. 

\begin{lemma}\label{noline}
Let $P \in \X(\Fq)$. If $\L_q (P, \X) = \emptyset$, then 
$|\X (\Fq)| < (d-1)q^3 + (d-1)q^2 + q + 1.$
\end{lemma}

\begin{proof}
In view of Lemma \ref{affine} and the fact $|\X(\Fq)| = |\X (\Fq) \cap T_P (\X)| + |\X(\Fq) \cap T_P (\X)^C|$, it is enough to show that $|\X (\Fq) \cap T_P (\X)| < (d-1)q^2 + q + 1$. Since $P$ is a singular point of $\X \cap T_P (\X)$, for each line $\ell$ with the property that  $P \in \ell \subset T_P (\X)$ 
we have $|\X \cap (\ell \setminus \{P\})| \le d-2$.
Since there are $q^2 + q + 1$ lines defined over $\Fq$ in $T_P (\X)$ that contain $P$, we have 
$$|\X (\Fq) \cap T_P (\X)| \le 1 + (d-2) (q^2 + q + 1) =  (d-2)q^2 + (d -2)q + d - 1 < (d-1)q^2 + q + 1.$$
This completes the proof.  
\end{proof}

\begin{definition}\normalfont
Let $P \in \X (\Fq)$ and  $\ell \in \L_q (P, \X)$. For each $Q \in \ell(\Fq)$ we define,  
$$\Omega_{\ell} (Q) := \{ \Pi \in \B(\ell) \mid \X \cap \Pi = \ell \cup \ell_1 \cup \dots \cup \ell_{d-1}, \ell_i \in \L_q (Q, \X) \} \ \ \mathrm{and} \ \ \Omega (\ell) := \displaystyle{\bigcup_{Q \in \ell (\Fq)}} \Omega_{\ell} (Q).$$
\end{definition}

\begin{lemma}\label{nocone}
Let $P \in \X(\Fq)$ and suppose that there exists $\ell \in \L_q(P, \X)$ such that for any $Q \in \ell (\Fq)$ the surface $\X \cap T_Q (\X)$ does not contain a cone over a plane curve defined over $\Fq$ with center at $Q$. Then 
$$|\Omega_{\ell} (Q)| \le d-1 \ \text{for all} \ Q \in \ell (\Fq), \ \text{and  consequently,}   \ |\Omega (\ell)| \le (d-1)(q+1).$$
\end{lemma}

\begin{proof}
Let $Q \in \ell (\Fq)$ and  $t_Q = |\Omega_{\ell} (Q)|$. This implies that there are $t_Q$ planes $\Pi_1, \dots, \Pi_{t_Q}$ each defined over $\Fq$ containing $d-1$ lines other than $\ell$ defined over $\Fq$ passing through $Q$. Then $|\L_q (Q, \X)| \ge (d-1)t_Q + 1$. If $t_Q \ge d$, then  
$|\L_q (Q, \X)| \ge  d(d-1) + 1 >  d(d-1)$.
This contradicts Theorem \ref{three}. The second inequality follows since $|\Omega (\ell)| \le \sum_{P \in \ell (\Fq)} |\Omega_{\ell} (P)| \le (d-1)(q+1).$ 
\end{proof}

\begin{remark}\label{3} \normalfont
Let $P \in \X(\Fq)$ and suppose that there exists $\ell \in \L_q(P, \X)$ such that for any $Q \in \ell (\Fq)$ the surface $\X \cap T_Q (\X)$ does not contain a cone over a plane curve defined over $\Fq$ with center at $Q$. We define,
$$\S (\ell) := \{\Pi \in \B (\ell) \mid \X \cap \Pi \ \text{is a union of} \ d \ \text{lines defined over} \ \Fq\}.$$
Using Theorem \ref{three} and a similar argument as in the proof of Lemma \ref{nocone} it is easy to show that $|\S (\ell)| \le d (q+1) -1.$ In the special case when $d = 3$ and $\Pi \in \B (\ell) \setminus \S (\ell),$ then $\X \cap (\Pi \setminus \ell)$ is a plane curve of degree $2$ defined over $\Fq$. Furthermore, $\X \cap (\Pi \setminus \ell)$ does not contain a line defined over $\Fq$. Using Theorem \ref{szik}, we may conclude that $|\X \cap (\Pi \setminus \ell)| \le q +1$.
\end{remark}

\begin{lemma}
For $\Pi \in \B(\ell)$, we have 
$$|\X(\Fq) \cap (\Pi \setminus \ell)|  =(d-1)q \ \text{if} \ \Pi \in \Omega (\ell) \ \ \text{and}  \ \ |\X(\Fq) \cap (\Pi \setminus \ell)|\le (d-1)q - (d-2)  \ \text{if} \ \Pi \in \B(\ell) \setminus \Omega (\ell)
.$$
\end{lemma}

\begin{proof}
If $\Pi \in \Omega(\ell)$ then we see from direct computation that $|\X (\Fq) \cap \Pi| = dq + 1$. We note that $|\X (\Fq) \cap (\Pi \setminus \ell)| = |\X (\Fq) \cap \Pi| - |\ell| = (d-1)q$, which proves the first assertion.
To prove the second assertion, choose $\Pi \in \B(\ell) \setminus \Omega (\ell)$. It follows readily that $\X \cap \Pi$ is not a union of $d$ lines with a point in common. From the second part of Theorem \ref{SS} we have $|\X (\Fq)| < dq + 1$. Moreover, $\X \cap (\Pi \setminus \ell) $ is an affine curve of degree $d - 1$ defined over $\Fq$ with
$|\X(\Fq) \cap (\Pi \setminus \ell)| = |\X(\Fq) \cap \Pi| - |\ell (\Fq)| < dq +1 - (q+1) = (d-1)q$. Since $d- 1 \le q-1$,  Theorem \ref{O} (b) applies, and we obtain $|\X(\Fq) \cap (\Pi \setminus \ell)| \le (d-1)q - (d-2)$. 
\end{proof}

\begin{theorem}\label{main}
Fix a positive integer $d$ with $2 \le d \le q$. Let $\X \subset \PP^4$ be a nonsingular threefold of degree $d$ defined over $\Fq$. If $(d, q) \neq (4, 4)$ we have,
$$|\X(\Fq)| \le (d-1)q^3 + (d-1)q^2 + q +1.$$
Moreover, the bound is attained by a nonsingular threefold $\X$ of degree $d$ only if there exists a point $P \in \X (\Fq)$ such that $\X \cap T_P (\X)$ is a cone, with center at $P$, over a plane curve $\C$ of degree $d$ defined over $\Fq$ that does not contain a line defined over $\Fq$ and $|\C (\Fq)| = (d-1)q + 1.$ 
\end{theorem}

\begin{proof}
If $d = 2$, then Proposition \ref{quad} applies and proves the assertion. Thus, we may assume that $d \ge 3$. If $\X(\Fq) = \emptyset$ there is nothing to prove. Choose $P \in \X (\Fq)$. If $\L_q (P, \X) = \emptyset$ then the Theorem is proved using Lemma \ref{noline}. Thus, we may assume that $\L_q(P, \X) \neq \emptyset$. Let $\ell \in \L_q (P, \X)$. We divide the proof into various cases. 

\emph{Case 1: There exists  $Q \in \ell (\Fq)$ such that $\X \cap T_Q (\X)$ contains a cone over a plane curve $\C$ defined over $\Fq$ with center at $Q$.} Suppose that there exists a plane $\Pi$ defined over $\Fq$ such that $\C \subset \Pi$. Let $\deg \C = d_1$.   We note that $\C$ does not contain a line defined over $\Fq$, for otherwise $\X \cap T_Q(\X)$ would contain a plane defined over $\Fq$, contradicting Proposition \ref{kt}. Since $d_1 \le d$ and $(d, q) \neq (4, 4)$, we have $(d_1, q) \neq (4, 4)$. From Theorem \ref{szik} we have $|\C (\Fq)| \le (d_1 - 1)q + 1$ and consequently $|\C^* (\Fq)| \le (d_1-1)q^2 + q + 1$, where $\C^*$ denotes the cone over $\C$ with center at $P$. If $d_1 = d$, then $\X \cap T_Q (\X) = \C^*$. We have $|\X(\Fq) \cap T_Q (\X)| \le (d-1)q^2 + q + 1$ and the assertion is proved using Lemma \ref{affine}. If $d_1 < d$, then there exists a surface $Z$ of degree at most $d-d_1$ such that $\X \cap T_Q (\X) = \C^* \cup Z$.  Note that $Z$ does not contain any plane defined over $\Fq$.  Using Theorem \ref{elementary} we have $|Z(\Fq)| \le (d - d_1 - 1)q^2 + (d- d_1)q + 1$. Hence,
$$|\X (\Fq) \cap T_Q (\X)| \le |\C^* (\Fq)| + |Z (\Fq)| \le (d - 1) q^2 + q + 1 + \big(-q^2 + (d - d_1)q + 1 \big).$$
Since $d_1 \ge 1$, we deduce that $q^2 > (d-d_1)q + 1$. This shows that $|\X (\Fq) \cap T_Q (\X)| <  (d - 1) q^2 + q + 1$. The assertion of the theorem is now proved using Lemma \ref{affine}.

\emph{Case 2: For each  $Q \in \ell (\Fq)$ the corresponding surface $\X \cap T_Q (\X)$ does not contain a cone over plane curve defined over $\Fq$ with center at $Q$.} Let $\Omega (\ell)$ and $\S (\ell)$ be as above. We first assume that $(d, q) \neq (3, 3).$
 Following the notations above, let $r = |\B(\ell) \setminus \Omega (\ell)|$. From Lemma \ref{nocone} we have $r \ge (q^2 + q + 1) - (d-1)(q+1)$. Also, from Theorem \ref{O}, we derive that 
$|\X (\Fq)\cap (\Pi \setminus \ell)|  = (d-1)q$ if $\Pi \in \Omega(\ell)$ and $|\X (\Fq)\cap (\Pi \setminus \ell)|  \le (d-1)q - (d-2)$ if $\Pi \in \B(\ell) \setminus \Omega(\ell)$. Hence
\begin{align*}
|\X(\Fq)| &= |\ell (\Fq)| + \sum_{\Pi \in \B(\ell)} |\X (\Fq) \cap (\Pi \setminus \ell)| \\
&= |\ell (\Fq)| + \sum_{\Pi \in \Omega(\ell)} |\X (\Fq)\cap (\Pi \setminus \ell)| + \sum_{\Pi \in \B(\ell) \setminus \Omega(\ell)} |\X (\Fq) \cap (\Pi \setminus \ell)| \\
&\le q+1 +(q^2 + q + 1 - r) (d-1)q +  r \big((d-1)q - (d-2) \big) \\
&=(d-1)q^3 + (d-1)q^2 + q + 1 + \big((d-1)q - r (d-2) \big).  
\end{align*}
To prove the assertion it is enough to show that $r (d-2) - (d-1)q > 0$. 

\emph{Subcase 1:} Let $d \le q-1$. We have $r \ge (q^2 + q + 1) - (q^2 - q - 2) = 2q + 3$, implying   
$$(d-2)r - (d-1)q \ge (d-2)(2q + 3) - (d-1)q = (d -3)q + 3 (d-2) > 0,$$
the last inequality follows since $d \ge 3$. 

\emph{Subcase 2:} Let $d = q$. Then $r \ge (q^2 + q + 1) - (q^2 - 1) = q+2$. Thus,
$$(d-2)r - (d-1)q \ge (q-2)(q + 2) - (q-1)q = q^2 - 4 - q^2 + q > 0,$$
this follows since $d \ge 3$ and $(d, q)\neq (3, 3)$. Furthermore, the strict inequality in subcase 2 is a consequence of the fact that $(d, q) \neq (4, 4)$. 

To deal with the case $(d, q) = (3, 3)$, we would need a better estimate. To this end, let $r = |\B(\ell) \setminus \Omega (\ell)|$ as above, and define $s = |\B(\ell) \setminus \S (\ell)|$, where $\S (\ell)$ is as in Remark \ref{3}. It turns out that, $s \ge q^2 + q + 1 - d (q + 1) + 1 = q^2 - (d-1)(q+1) + 1$. In particular, for $(d, q) = (3, 3)$, we have $r \ge 5$ and $s \ge 2$. We also note that $\Omega (\ell) \subset \S (\ell)$ and consequently, $\B (\ell) = \Omega (\ell) \sqcup (\S (\ell) \setminus \Omega (\ell)) \sqcup (\B (\ell) \setminus \S (\ell))$. We have,
\begin{align*}
& \ \ \ |\X(\Fq)| \\
&= |\ell (\Fq)| + \sum_{\Pi \in \B(\ell)} |\X (\Fq) \cap (\Pi \setminus \ell)| \\
&= |\ell (\Fq)| + \sum_{\Pi \in \Omega(\ell)} |\X (\Fq)\cap (\Pi \setminus \ell)| + \sum_{\Pi \in \S(\ell) \setminus \Omega(\ell)} |\X (\Fq) \cap (\Pi \setminus \ell)| + \sum_{\Pi \in \B(\ell) \setminus \S(\ell)} |\X (\Fq) \cap (\Pi \setminus \ell)|\\
&\le q+1 +(q^2 + q + 1 - r) 2q +  (r -s) (2q - 1 ) + s (q+1)\\
&=2q^3 + 2q^2 + q + 1 + \big( 2q (1 - r) + (r - s) (2q -1) + s (q+1) \big).  
\end{align*}
It is enough to prove that $2q (1 - r) + (r - s) (2q -1) + s (q+1) < 0$. But for $q = 3$, we have $2q (1 - r) + (r - s) (2q -1) + s (q+1) = 6 - (r + s) < 0$.
This completes the proof of the first assertion. The second assertion is follows from the proof of the first assertion.
\end{proof}

\begin{remark}\normalfont
 As we have observed, the upper bound in the Theorem \ref{main} is always attained by a nonsingular quadric threefold. It is well-known that a nonsingular Hermitian threefold also attains this bound (see \cite{BC} for more on Hermitian varieties in general). Further, if there is a plane curve $\C$ of degree $d$ not containing a line defined over $\Fq$ such that $|\C (\Fq)| = (d-1)q + 1$, then  $d = 2$ or $d \ge \sqrt{q} + 1$ (see \cite[Lemma 2.3]{HK4}). From the second assertion of Theorem \ref{main} it is now clear that the upper bound is attained by a nonsingular threefold of degree $d$ defined over $\Fq$ only if $d = 2$ or $d \ge \sqrt{q} + 1$.
\end{remark}

We conclude this article by comparing the upper bound obtained in Theorem \ref{main} to the upper bounds mentioned in Theorem \ref{SS}, Theorem \ref{elementary}, and in Theorem \ref{del}.

\begin{remark}\normalfont
We denote by $\mathscr{B} (d) := (d -1)q^3 + (d-1)q^2 + q + 1$. We also have,
$$ \mathscr{S} = dq^3 + q^2 + q + 1 \ \ \ \text{and} \ \ \ \mathscr{E} (d, 4) = (d-1)q^3 + dq^2 + q +1.$$ 
A direct comparison shows that $\mathscr{B} (d) < \mathscr{E} (d, 4) < \mathscr{S} (d, 4)$. As pointed out in Remark \ref{elt}, we have $\mathscr{E} (d, m) < \mathscr{W} (d, m) $ whenever $m \ge 3$ and $d \ge \sqrt{q} + 2$. In particular, this implies that the upper bound $\mathscr{B} (d)$ is better than $\mathscr{W} (d, 4)$ whenever $d \ge \sqrt{q} + 2$.
\end{remark}

\section*{Acknowledgment}
The author expresses his gratitude to the anonymous referee for their careful reading and relevant suggestions. Thanks are due to Peter Beelen for some discussions and comments on this article. 

\end{document}